\documentclass[10pt]{amsart} %

\usepackage{float}

\usepackage{epsfig}
\usepackage{epstopdf}
\usepackage{tikz}
\usepackage{array,amsmath, enumerate,  url, psfrag}
\usepackage{amssymb, amsaddr, fullpage}
\usepackage{graphicx,subfigure}
\usepackage{color}

\newtheorem{theorem}{Theorem}[section]
\newtheorem{lemma}[theorem]{Lemma}

\newtheorem{proposition}[theorem]{Proposition}
\newtheorem{corollary}[theorem]{Corollary}

\newcommand{\n}{\noindent}
\newcommand{\pend}{\text{end}}
\newcommand{\pint}{\text{int}}
\newcommand{\com}{\text{,}\,\allowbreak}
\newtheorem{thm}{Theorem}[section]

\newtheorem{claim}[thm]{Claim}

\newtheorem{definition}[thm]{Definition}

\title{Fat-triangle linkage and kite-linked graphs}

\author{Runrun Liu$^{1}$,  Martin Rolek$^{2}$,    Gexin Yu$^{2}$}

\address{
$^{1}$\small School of Mathematics and Statistics, Central China Normal University, Wuhan, Hubei, China.\\
$^2$\small Department of Mathematics, The College of William and Mary, Williamsburg, VA, 23185, USA.
}

\thanks{The work is done while the first author is studying at the College of William and Mary as a visiting student,  supported by the Chinese Scholarship Council.  The research of the last author was supported in part by the Natural Science Foundation of China (11728102) and the NSA grant H98230-16-1-0316.}

\email{827261672@qq.com (R. Liu), msrolek@wm.edu, gyu@wm.edu}

\begin{document}
\maketitle

\begin{abstract}
For a multigraph $H$, a graph $G$ is $H$-linked if every injective mapping $\phi: V(H)\to V(G)$ can be extended to an $H$-subdivision in $G$.
We study the minimum connectivity required for a graph to be $H$-linked.
A $k$-fat-triangle $F_k$ is a multigraph with three vertices and a total of $k$ edges.
We determine a sharp connectivity requirement for a graph to be $F_k$-linked.
In particular, any $k$-connected graph is $F_k$-linked when $F_k$ is connected.
A kite is the graph obtained from $K_4$ by removing two edges at a vertex.
As a nontrivial application of $F_k$-linkage, we then prove that every $8$-connected graph is kite-linked, which shows that the required connectivity for a graph to be kite-linked is $7$ or $8$.
\end{abstract}

\section{Introduction}

In graph theory, we often need to find structures with certain constraints.
For example, a graph is $k$-connected if and only if for every pair of $k$-vertex sets $S$ and $T$, there exist $k$ disjoint paths from $S$ to $T$.
In this case, though, we have limited control over the endpoints of such paths.
It would be helpful to know if such disjoint paths can still be found when we specify the endpoints of each of the paths.
This leads to the notion of $k$-linked graphs.
A graph is {\em $k$-linked} if, for any $2k$ distinct vertices $s_1, t_1, s_2, t_2, \ldots, s_k, t_k$, there are disjoint paths $P_1, P_2, \ldots, P_k$ such that $P_i$ has endpoints $s_i$ and $t_i$ for each $i\in[k]$.
As another example, a classical result of Dirac~\cite{D60b} states that every $k$-connected graph has a cycle containing any given $k$ vertices.
What if we require the $k$ vertices to occur in the cycle in some given order?
This now leads to the notion of $k$-ordered graphs.
A graph is {\em $k$-ordered} if for every $k$ vertices with a given order, there exists a cycle containing the vertices in that order.

The $k$-linked graphs, $k$-ordered graphs, and other similar notions all have the same general flavor: for a given graph $H$, we try to find an $H$-subdivision (a subgraph which replaces the edges of $H$ by internally disjoint paths) in $G$ no matter how we place the vertices of $H$ in $G$.
This is the notion of $H$-linked graphs, first mentioned by Jung~\cite{J70}, and re-defined independently by Kostochka and G. Yu~\cite{KY05} and Ferrara, Gould, Tansey, and Whalen~\cite{FGTW06}.

\begin{definition}
Fix a multigraph $H$.
A graph $G$ is {\em $H$-linked} if for every injective mapping $\phi: V(H) \to V(G)$, there exists a mapping $\psi: E(H) \to \mathcal{P}(G)$, where $\mathcal{P}(G)$ is the set of paths in $G$, such that for every $uv \in E(H)$, $\psi(uv)$ is a $\phi(u), \phi(v)$-path in $G$, and distinct edges of $H$ map to internally disjoint paths in $G$.
\end{definition}

Note that an {\em $H$-subdivision} in a graph $G$ is a pair of mappings $(\phi, \psi)$ such that $\phi: V(H) \to V(G)$ is injective and $\psi$ satisfies the above condition.
So a graph is $H$-linked if every injective mapping $\phi: V(H) \to V(G)$ can be extended to an $H$-subdivision in $G$.
It is clear to see that a graph is $k$-connected if and only if it is $B_k$-linked, $k$-linked if and only if it is $kK_2$-linked, and $k$-ordered if and only if it is $C_k$-linked, where $B_k$ is the multigraph with two vertices and $k$ edges, $kK_2$ is a matching of size $k$, and $C_k$ is a cycle of length $k$.

Sufficient degree conditions for a graph to be $H$-linked have been extensively studied in \cite{FGTW06, GKY06, KKY06, KY05, KY08b, KY08a}.
In~\cite{LWY09}, Liu, West, and G. Yu asked to find the connectivity conditions for a graph to be $H$-linked.
For a given multigraph $H$, let $f(H)$ be the minimum positive integer $f$ such that every $f$-connected graph is $H$-linked.

When $H = kK_2$, we usually write $f(kK_2)$ as $f(k)$, which is the minimum positive integer $f$ such that every $f$-connected graph is $k$-linked.
This is a well-studied parameter, see Jung~\cite{J70}, Larman and Mani~\cite{LM70}, Mader~\cite{M67}, Robertson and Seymour~\cite{RS95}, Bollob\'as and Thomason~\cite{BT96}, and Kawarabayashi, Kostochka, and G. Yu~\cite{KKY06}.
Thomas and Wollan~\cite{TW05} have the best current general bound for $f(k)$, namely that $f(k) \le 10k$.

Better bounds for $f(k)$ have been found for some small values of $k$.
Jung~\cite{J70} showed that any $4$-connected non-planar graph is $2$-linked.
Seymour~\cite{S80}, and independently Thomassen~\cite{T80}, gave a complete characterization of non-$2$-linked graphs.
It follows that $f(2) = 6$.
Thomas and Wollan~\cite{TW08} improved the connectivity of their result in~\cite{TW05} to show that $10$-connected graphs are $3$-linked.
More specifically, they prove that $6$-connected graphs on $n$ vertices with at least $5n-14$ edges are $3$-linked and the edge bound is sharp.

Clearly, $f(H) \le f(k)$ if $H$ has at most $k$ edges and no isolated vertices.
X. Yu~\cite{Y03} charaterized all obstructions to $P_4$-linked graphs, and showed that $7 \le f(P_4) \le 8$, and Ellingham, Plummer, and G. Yu~\cite{EPY12} showed that $f(P_4) = 7$.
Very recently, McCarty, Wang, and X. Yu~\cite{MWY18} showed that $f(C_4) = 7$, that is, $7$-connected graphs are $4$-ordered, confirming a conjecture of Faudree~\cite{F01}.  Beyond this, very little is known about $f(H)$ for other multigraphs $H$.


In this paper, we are able to determine $f(H)$ exactly for another class of multigraphs $H$, namely, fat-triangles.
A graph $H = F_{k_1, k_2, k_3}$ is called a \emph{$(k_1, k_2, k_3)$-fat-triangle} if $|V(H)| = 3$, say $V(H) = \{v_1, v_2, v_3\}$, and there exist $k_i$ edges joining $v_i$ and $v_{i + 1}$ for each $i \in [3]$ (with subscripts read modulo 3).
Note that by a simple application of Menger's Theorem, any $(k + 1)$-connected graph is $F_{k, 0, 0}$-linked.
We will show that if $F_{k_1, k_2, k_3}$ is connected, then $(k_1 + k_2 + k_3)$-connected graphs are $F_{k_1, k_2, k_3}$-linked.

\begin{theorem}\label{thm:fattriangle}
Let $G$ be a $k$-connected graph.  Let $k_1, k_2, k_3$ be integers such that $k_1, k_2 > 0$, $k_3 \ge 0$, and $k_1 + k_2 + k_3 \le k$.  Then $G$ is $F_{k_1, k_2, k_3}$-linked.
\end{theorem}

Note that the connectivity condition in Theorem~\ref{thm:fattriangle} is sharp.
Let $k_1, k_2, k_3$ be as in the statement of Theorem~\ref{thm:fattriangle}, and suppose $k_1 + k_2 + k_3 = k$.
Let $G$ be a $(k - 1)$-connected graph with a separating set $S \subseteq V(G)$ with $|S| = k -1$ such that $G - S$ has three distinct components $C_1, C_2, C_3$.
Let $v_i \in V(C_i)$ for $i \in [3]$.
Then for distinct $i, j \in [3]$, every $v_i, v_j$-path must use at least one vertex of $S$.
Therefore $G$ cannot be $F_{k_1, k_2, k_3}$-linked.

\begin{corollary}
Let $k_1, k_2, k_3$ be integers with $k_1, k_2 > 0$, $k_3 \ge 0$.  Then $f(F_{k_1, k_2, k_3}) = k_1 + k_2 + k_3$.
\end{corollary}

We know that every $3$-connected graph contains a cycle through any three given vertices in the graph.
Our Theorem~\ref{thm:fattriangle} on fat-triangle-linkage generalizes this result.
There is another important reason for us to study fat-triangle-linkage.
In the study of $H$-linkage, it is important to find suitable intermediate substructures to work with.
For example, X. Yu~\cite{Y03} used ``ladders'' to study $P_4$-linkage, and McCarty, Wang, and X. Yu~\cite{MWY18} used ``skeletons'' to study $C_4$-linkage.
A fat-triangle is one such substructure that can be useful in the study of some $H$-linkage problems, in particular when $H$ has few vertices or has a $K_3$ subgraph.

We mentioned above that the only known exact values for $f(H)$ when $H$ has more than two edges are $f(B_k)$, $f(P_4)$, $f(C_3)$, and $f(C_4)$ (and now also fat-triangles).
Let a {\em kite}, denoted $P_4^+$, be the subgraph obtained from $K_4$ by removing two edges at a vertex.
The kite is the only connected graph $H$ with four vertices and at most four edges for which $f(H)$ remains unknown.
As $P_4$ is a subgraph of a kite, kite-linked graphs are also $P_4$-linked.
Therefore $f(P_4^+) \ge 7$.
On the other hand, a kite has four edges, and so $4$-linked graphs are $P_4^+$-linked, giving $f(P_4^+) \le 40$.
With the help of fat-triangle-linkage, we are able to narrow down $f(P_4^+)$ to two possible values.



\begin{theorem}\label{kite}
All $8$-connected graphs are kite-linked. Consequently, $f(P_4^+)\in \{7,8\}$.
\end{theorem}

Note that the proof of $f(C_4) = 7$ in \cite{MWY18} relies on the result $f(P_4) = 7$ from~\cite{EPY12}, whose proof was quite involved.
Our proof for kite-linkage is self-contained and follows the ideas of the proof of $f(C_4) = 7$ in \cite{MWY18}.
By using the fat-triangle linkage, we first find a structure called a flower (see Figure~\ref{flower}), and show that every $7$-connected graph with a certain specified flower is kite-linked.
We do need $8$-connectedness to guarantee the existence of the desired flower though.
By doing something similar to our proof, one could find a self-contained and much simpler proof for $f(P_4) \le 7$.
We discuss this in more detail in the Final Remarks section.

\section{Connectivity for fat-triangle-linkage}

In this section, we prove Theorem~\ref{thm:fattriangle}.  Our proof uses the following Theorem~\ref{thm:HWege}, which is commonly referred to as Mader's $H$-Wege Theorem, or $S$-paths Theorem.  The statement of Theorem~\ref{thm:HWege} is actually a slight modification of Mader's original result~\cite{M78} given by Robertson, Seymour, and Thomas~\cite{RST93}.
An alternate, shorter proof of Mader's original theorem has been given by Schrijver~\cite{S01}.
Let $L_1, \dots, L_t$ be distinct subsets of $V(G)$.
We say a path in $G$ is \emph{good} if it has ends $u, v$ with $u \in L_i$ and $v \in L_j$ for $i \ne j$.

\begin{theorem}[Mader~\cite{M78}; Robertson, Seymour, and Thomas~\cite{RST93}]\label{thm:HWege}
Let $G$ be a graph, let $L_1, \dots, L_t$ be subsets of $V(G)$, and let $k \ge 0$ be an integer.
Then exactly one of the following holds:
\begin{enumerate}[(i)]
\item there are $k$ good paths of $G$, mutually vertex-disjoint, or

\item there exist a partition $W, Y_1, \dots Y_n$ of $V(G)$, and for all $j \in [n]$ a subset $X_j \subseteq Y_j$, such that
\begin{enumerate}[(a)]
\item $|W| + \sum_{j = 1}^n \left\lfloor \frac{1}{2} |X_j| \right\rfloor < k$,

\item for all $j \in [n]$, no vertex in $Y_j-X_j$ has a neighbor in $V(G)-(W \cup Y_j)$, and $Y_j \cap L_i \subseteq X_j$ for all $i \in [t]$, and

\item every good path $P$ in $G$ with $V(P) \cap W = \emptyset$ has an edge with both ends in $Y_j$ for some $j \in [n]$.
\end{enumerate}\end{enumerate}
\end{theorem}

\n{\bf Proof of Theorem~\ref{thm:fattriangle}.}
We may assume $k_1 + k_2 + k_3 = k$, for otherwise we may let $k_1' = k - (k_2 + k_3)$ and find a $(k_1', k_2, k_3)$-fat triangle linkage, which contains a $(k_1, k_2, k_3)$-fat-triangle linkage as a subgraph.
Let $v_1, v_2, v_3 \in V(G)$ be distinct.
Suppose first that $k_3 = 0$.
Then we must find $k_1$ paths with ends $v_1, v_2$ and $k_2$ paths with ends $v_2, v_3$, such that all paths are internally disjoint.
Let $G'$ be the graph obtained from $G$ by duplicating $v_1$ $k_1$ times and duplicating $v_3$ $k_2$ times.
Then $G'$ is $k$-connected, so there exist $k$ disjoint paths from $v_2$ to the set of copies of $v_1$ and $v_3$, disjoint except for their common end $v_2$.
Each such path corresponds to a path in $G$ with one end $v_2$ and the other end either $v_1$ or $v_3$.
These paths in $G$ are internally disjoint and give the required $(k_1, k_2, 0)$-fat triangle linkage.

Thus we may assume $k_3 > 0$.
We now proceed by induction on $k$.
It is well-known that any $3$-connected graph contains a cycle through any three of its vertices, and is thus $(1, 1, 1)$-fat-triangle linked.
Hence $k \ge 4$ and we may assume that any $k'$-connected graph is $(k_1', k_2', k_3')$-fat-triangle linked, where $3 \le k' < k$, $1 \le k_i' \le k_i$ for $i \in [2]$, $0 \le k_3' \le k_3$, and $k_1' + k_2' + k_3' = k'$.

\begin{claim}\label{fat:ShortPath}
For all distinct $i, j \in [3]$, we may assume there is no $v_i, v_j$-path $P$ in $G$ with $|V(P)| \le 3$.
\end{claim}
\begin{proof}
Suppose $P$ is such a path in $G$, chosen so that $|V(P)|$ is minimum.
By relabeling if necessary, we may assume $P$ has ends $v_3, v_1$.
If $|V(P)| = 2$, define $G' := G - v_1v_3$, and if $|V(P)| = 3$ let $w$ be the interior vertex of $P$, and define $G' := G - w$.
Note by the choice of $P$ that $w \ne v_2$.
Thus, in either case, $G'$ is $(k - 1)$-connected, so there exists a $(k_1, k_2, k_3 - 1)$-fat triangle linkage at $v_1, v_2, v_3$.
Such a linkage in $G'$ is easily extended to a $(k_1, k_2, k_3)$-fat triangle linkage in $G$ by adding the path $P$.
\end{proof}

Now let $G'$ be the graph obtained from $G$ by duplicating $v_1$ $k_1 + k_3$ times, duplicating $v_2$ $k_1 + k_2$ times, and duplicating $v_3$ $k_2 + k_3$ times.
Let $L_i \subseteq V(G')$ be the set of copies of $v_i$ for $i \in [3]$.
It is easy to see that if $G'$ has $k$ good paths, then it must have $k_1$ $L_1, L_2$-paths, $k_2$ $L_2, L_3$-paths, and $k_3$ $L_3, L_1$-paths, and these good paths in $G'$ correspond with paths in $G$ that give the desired $(k_1, k_2, k_3)$-fat-triangle linkage.
Therefore we may assume that $G'$ does not have $k$ good paths.
By Theorem~\ref{thm:HWege}, there exists $W \subseteq V(G')$, a partition $Y_1, \dots Y_n$ of $V(G') \setminus W$, and for all $j \in [n]$ a subset $X_j \subseteq Y_j$, such that they satisfy (a)(b)(c) of Theorem~\ref{thm:HWege}.




We may assume that the sets $W, Y_1, \dots, Y_n, X_1, \dots, X_n$ are chosen such that $W$ is maximal, and we may further assume that $Y_i \ne \emptyset$ for all $i \in [n]$.
We now prove a series of claims which establish the structure of the graph $G'$.

\begin{claim}\label{fat:n2} $n \ge 2$.
\end{claim}
\begin{proof}
If $n = 1$, then $L_1 \cup L_2 \cup L_3 \subseteq W \cup X_1$.
Since $|L_1 \cup L_2 \cup L_3| = (k_1 + k_3) + (k_1 + k_2) + (k_2 + k_3) = 2k$, we have $|W| + \left\lfloor \frac{1}{2} |X_1| \right\rfloor \ge k$, contradicting (a).
\end{proof}

\begin{claim}\label{fat:Xjempty} $X_j \ne \emptyset$ for any $j \in [n]$.
\end{claim}
\begin{proof}
Assume to the contrary that $X_1 = \emptyset$.
Since $Y_1 \ne \emptyset$, and by (b) no vertex of $Y_1$ has a neighbor in $V(G') \setminus (W \cup Y_1)$, it follows that $W$ separates $Y_1$ and $Y_2$, where $Y_2 \ne \emptyset$ by Claim~\ref{fat:n2}.
Since $G'$ is $k$-connected, we must have $|W| \ge k$, contradicting (a).
\end{proof}

\begin{claim}\label{fat:Xjodd} $X_j$ is odd for all $j \in [n]$.
\end{claim}
\begin{proof}
If, say, $|X_1|$ is even, then $|X_1| \ge 2$ by Claim~\ref{fat:Xjempty}.
Let $x \in X_1$.
Define $W' = W \cup \{x\}$, $Y_1' = Y_1 \setminus \{x\}$, $X_1' = X_1 \setminus \{x\}$, and $Y_j' = Y_j$ and $X_j' = X_j$ for all $j \in \{2, \dots, n\}$.
Then $W', Y_1', \dots, Y_n', X_1', \dots, X_n'$ satisfy (a)-(c), contrary to our choice of $W$ as maximal.
\end{proof}

%
%
%
%

\begin{claim}\label{fat:2Xj}
Every good path in $G'$ which avoids $W$ has at least 2 vertices in $X_j$ for some $j \in [n]$.
\end{claim}
\begin{proof}
Suppose $Q$ is a path in $G' \setminus W$ with ends $u_1 \in L_1$ and $u_2 \in L_2$, say.
By (c), $Q$ has some edge $e$ with both ends in $Y_j$ for some $j \in [n]$.
By (b), the subpath of $Q$ from $u_1$ to the first end of $e$ must contain some vertex of $X_j$, and the subpath of $Q$ from the second end of $e$ to $u_2$ must also contain some vertex of $X_j$.
%
\end{proof}

%

\begin{claim}\label{fat:W0}
$W = \emptyset$.
\end{claim}
\begin{proof}
Suppose to the contrary that $|W| = m$ for some $m \in [k - 1]$.
Let $M$ be a maximum collection of (unordered) pairs $(u, v)$ such that $u \in L_i$ and $v \in L_{i'}$ for distinct $i, i' \in [3]$, and no vertex of $L_1 \cup L_2 \cup L_3$ belongs to more than one pair in $M$.
Then $|M| = k$.
Consider the set $M' = \{(u, v) \in M : u, v \notin W\}$.
Then $|M'| \ge k - m$.
For $i \in [3]$, let $k_i' = |\{(u, v) \in M' : u \in L_i, v \in L_{i + 1}\}|$, where subscripts are read mod 3.
Then $k_1' + k_2' + k_3' \ge k - m$.
If, say, $k_2' = k_3' = 0$, then since $G' \setminus W$ is $(k - m)$-connected, there exist $k - m$ disjoint $L_1, L_2$-paths in $G' \setminus W$.
These paths are good paths in $G'$, so by Claim~\ref{fat:2Xj}, $|W| + \sum_{j = 1}^n \left\lfloor \frac{1}{2} |X_j| \right\rfloor \ge m + \frac{1}{2}(2(k - m)) = k$, contradicting (a).
Hence we may assume $k_1', k_2' \ne 0$.
Then let $k_1^*, k_2^*, k_3^*$ be chosen so that $1 \le k_1^* \le k_1'$, $1 \le k_2^* \le k_2'$, $0 \le k_3^* \le k_3'$, and $k_1^* + k_2^* + k_3^* = k - m$.
Note that $W \setminus (L_1 \cup L_2 \cup L_3)$ is a subset of $V(G)$.
Then since $G \setminus (W \setminus (L_1 \cup L_2 \cup L_3))$ is at least $(k - m)$-connected, there exists a $(k_1^*, k_2^*, k_3^*)$-fat-triangle linkage at $v_1, v_2, v_3$ by induction.
Any such linkage in $G \setminus (W \setminus (L_1 \cup L_2 \cup L_3))$ corresponds with $k_1^* + k_2^* + k_3^* = k - m$ good paths in $G'$ by using pairs in $M'$ as the ends of the paths.
So again by Claim~\ref{fat:2Xj}, $|W| + \sum_{j = 1}^n \left\lfloor \frac{1}{2} |X_j| \right\rfloor \ge k$, contradicting (a).
\end{proof}

\begin{claim}\label{fat:Yj=Xj}
For all $j \in [n]$, if $|X_j| < k$ then $Y_j = X_j$.
\end{claim}
\begin{proof}
If $Y_j \setminus X_j \ne \emptyset$, then by (b) and Claims~\ref{fat:n2} and~\ref{fat:W0}, $X_j$ is a separating set.
Since $G$ is $k$-connected, we must have $|X_j| \ge k$.
\end{proof}

\begin{claim}\label{fat:Xj3}
For all $j \in [n]$, $|X_j| \ge 3$.
\end{claim}
\begin{proof}
Since $G'$ is connected, $G'$ has at least one good path which avoids $W$ since $W = \emptyset$ by Claim~\ref{fat:W0}.
Thus by Claims~\ref{fat:Xjodd} and~\ref{fat:2Xj}, $|X_j| \ge 3$ for some $j \in [n]$.
Without loss of generality, we may assume there is some $m \in \{2, 3, \dots, n - 1\}$ such that $|X_j| \ge 3$ for all $j \in [m]$ and $|X_j| = 1$ for all $j \in \{m + 1, \dots, n\}$.
For convenience, let $U = X_{m + 1} \cup X_{m + 2} \cup \dots \cup X_n$.
For $i \in [3]$, define the set $Z_i$ to be the union of the vertex sets of all paths $P$ meeting $L_i$ such that $P$ has no edge with both ends in $Y_j$ for any $j \in [n]$, and define $Z_0 = V(G') \setminus (Z_1 \cup Z_2 \cup Z_3)$.
Then for $i \in [3]$, $L_i \subseteq Z_i$.
We claim that $Z_1, Z_2, Z_3$ are pairwise disjoint.
For if $z \in Z_1 \cap Z_2$, say, then there exists some path $P$ from $L_1$ to $z$ and some path $Q$ from $L_2$ to $z$ such that neither $P$ nor $Q$ has an edge with both ends in $Y_j$ for any $j \in [n]$.
But then $P \cup Q$ is a good path in $G'$, no edge of which has both ends in the same set $Y_j$, contradicting (c).
Hence the sets $Z_0, Z_1, Z_2, Z_3$ are pairwise disjoint and have union $V(G')$.

Let $i \in \{0, 1, 2, 3\}$ such that $Z_i \cap U \ne \emptyset$.
Let $N$ be the set of vertices in $V(G') \setminus (Z_i \cap U)$ with a neighbor in $Z_i \cap U$.
Note that any vertex of $Z_i \cap U$ belongs to some set $X_j$ with $|X_j| = 1$, and so $Y_j = X_j$ by Claim~\ref{fat:Yj=Xj}.
Therefore any edge with one end in $Z_i \cap U$ and one end in $N$ does not have both ends in the same set $Y_j$, so it follows from the definition of the sets $Z_{i'}$ that both ends of such an edge belong to the same set $Z_{i'}$.
Hence $N \subseteq Z_i$.

Since $G'$ is $k$-connected, $|N| \ge k$.
We claim that $i = 0$.
So suppose that $i = 1$, say.
Then $|Z_1| \ge |N| \ge k$ and $|Z_2 \cup Z_3| \ge |L_2 \cup L_3| = (k_1 + k_2) + (k_2 + k_3) = k + k_2$.
Since $G'$ is $k$-connected, there exist $k$ disjoint $Z_1, Z_2 \cup Z_3$-paths.
By the definition of the sets $Z_i$ and since the sets $Z_i$ are disjoint, each of these paths must have some edge with both ends in the same set $Y_j$.
Thus, in a fashion similar to the proof of Claim~\ref{fat:2Xj}, it can be shown that each of these paths has at least two vertices in some set $X_j$.
Therefore, $\sum_{j = 1}^n \left\lfloor \frac{1}{2} |X_j| \right\rfloor \ge k$, contradicting (a).
Hence $i = 0$.

In particular, $U \subseteq Z_0$, and so $N$, $Z_1$, $Z_2$, and $Z_3$ are disjoint subsets of $X_1 \cup X_2 \cup \dots \cup X_m$.
Therefore $\sum_{j = 1}^m |X_j| \ge |N| + \sum_{i = 1}^3 |Z_i| \ge k + \sum_{i = 1}^3 |L_i| = 3k$.
Since $|X_j| \ge 3$ for $j \in \{1, \dots, m\}$, it follows that $\sum_{j = 1}^n \left\lfloor \frac{1}{2} |X_j| \right\rfloor \ge \frac{1}{3} \sum_{j = 1}^m |X_j| \ge k$, contradicting (a).
\end{proof}

\begin{claim}\label{fat:LiNbrs}
For $i \in [3]$, if there exist distinct $j, j' \in [n]$ such that $L_i \cap X_j, L_i \cap X_{j'} \ne \emptyset$, then $N(L_i) \subseteq (X_1 \cup \dots \cup X_n)$
\end{claim}
\begin{proof}
Say $u_1 \in L_1 \cap X_1$ and $u_2 \in L_1 \cap X_2$.
Since every vertex of $L_1$ is obtained by duplicating $v_1$, we have $N(u_1) = N(u_2) = N(L_1)$.
Thus by (b), any neighbor of $u_1$ in $Y_2$ belongs to $X_2$, any neighbor of $u_2$ in $Y_1$ belongs to $X_1$, and any neighbor of either $u_1$ or $u_2$ in $Y_3 \cup Y_4 \cup \dots \cup Y_n$ belongs to $X_3 \cup X_4 \cup \dots \cup X_n$.
\end{proof}

\begin{claim}\label{fat:2LiXj}
There exists $j \in [n]$ and distinct $i_1, i_2, i_3 \in [3]$ such that $(L_{i_1} \cup L_{i_2}) \subseteq X_j$ and $L_{i_3} \not\subseteq X_j$.
\end{claim}
\begin{proof}

First, suppose that $L_i \not\subseteq X_m$ for all $i \in [3]$ and $m \in [n]$.
By Claim~\ref{fat:ShortPath}, the sets $L_1, L_2, L_3, N(L_1), N(L_2)$, $N(L_3)$ are all disjoint.
Since $G'$ is $k$-connected, $|N(L_1)|\com |N(L_2)|\com |N(L_3)| \ge k$.
Then by Claim~\ref{fat:LiNbrs}, we have $|X_1 \cup \dots \cup X_n| \ge |L_1| + |L_2| + |L_3| + |N(L_1)| + |N(L_2)| + |N(L_3)| \ge 5k$.
By Claim~\ref{fat:Xj3}, $\sum_{j = 1}^n \left\lfloor \frac{1}{2} |X_j| \right\rfloor \ge \frac{1}{3}\sum_{j = 1}^n |X_j| \ge \frac{5}{3}k$, contradicting (a).

Thus we may assume that $L_1 \subseteq X_1$, say.
If $L_i \not\subseteq X_m$ for $i \in \{2, 3\}$ and $m \in [n]$, then by Claim~\ref{fat:LiNbrs}, $(N(L_2) \cup N(L_3)) \subseteq (X_1 \cup \dots \cup X_n)$.
Thus by Claim~\ref{fat:ShortPath} and since $G'$ is $k$-connected, $|(X_1 \cup X_2 \cup \dots \cup X_n) \setminus L_1| \ge |L_2| + |L_3| + 2k = 3k + k_2$.
Then
\begin{align*}
\sum_{j = 1}^n \left\lfloor \frac{1}{2} |X_j| \right\rfloor &\ge \frac{1}{2}(|L_1| - 1) + \frac{1}{3}(3k + k_2 - 1) \\
 &= \frac{4}{3}k + \frac{1}{6}(k_1 + k_3) - \frac{5}{6}
\end{align*}
Since $k_1, k_3 \ge 1$ and $k \ge 4$, we have $\frac{4}{3}k + \frac{1}{6}(k_1 + k_3) - \frac{5}{6} \ge k$, contradicting (a).

Therefore, we may additionally assume $L_2 \subseteq X_j$ for some $j \in [n]$.
Suppose $j \ne 1$, say $j = 2$.
If $N(L_i) \subseteq (X_1 \cup \dots \cup X_n)$ for $i \in \{1, 2\}$, then $\sum_{j = 1}^n \left\lfloor \frac{1}{2} |X_j| \right\rfloor \ge \frac{1}{3}(|L_1| + |L_2| + |L_3| + |N(L_1)| + |N(L_2)|) \ge \frac{4}{3}k$, contradicting (a).
Hence we may assume $N(L_1) \not\subseteq (X_1 \cup \dots \cup X_n)$.
Then $N(L_1) \cap (Y_1 \setminus X_1) \ne \emptyset$, so it follows from Claim~\ref{fat:Yj=Xj} that $|X_1| \ge k$.
If also $N(L_2) \not\subseteq (X_1 \cup \dots \cup X_n)$, then similarly $|X_2| \ge k$.
Then, if $X_3 \ne \emptyset$, it follows from Claim~\ref{fat:Xj3} that $\sum_{j = 1}^n \left\lfloor \frac{1}{2} |X_j| \right\rfloor \ge 2 \cdot \left\lfloor \frac{1}{2} k \right\rfloor + 1 \ge k$, contradicting (a).
So $n = 2$.
If $L_3 \subseteq X_1$ or $L_3 \subseteq X_2$, then we are done, so we may assume neither holds.
Thus by Claim~\ref{fat:LiNbrs}, $N(L_3) \subseteq X_1 \cup X_2$.
Then $|X_1 \cup X_2| \ge |L_1| + |L_2| + |L_3| + |N(L_3)| \ge 3k$, and so $\sum_{j = 1}^n \left\lfloor \frac{1}{2} |X_j| \right\rfloor \ge \frac{1}{2}(3k - 2) \ge k$ since $k \ge 4$, contradicting (a).
Thus we may assume that $N(L_2) \subseteq (X_1 \cup \dots \cup X_n)$.
Either by similar argument if $L_3 \subseteq X_j$ for some $j \in \{3, \dots, n\}$ or by Claim~\ref{fat:LiNbrs}, we may also assume $N(L_3) \subseteq (X_1 \cup \dots \cup X_n)$.
Then
\begin{align*}
\sum_{j = 1}^n \left\lfloor \frac{1}{2} |X_j| \right\rfloor &\ge \frac{1}{2}(|L_1| - 1 + |L_2| - 1) + \frac{1}{3}(|L_3| + |N(L_2)| + |N(L_3)| - 2) \\
 &= \frac{4}{3}k + \frac{1}{6}(2k_1 + k_2 + k_3) - \frac{5}{3}
\end{align*}
Now since $k_1, k_2, k_3 \ge 1$ and $k \ge 4$, we have $\frac{4}{3}k + \frac{1}{6}(2k_1 + k_2 + k_3) - \frac{5}{3} \ge k$, contradicting (a).

Therefore we must have $(L_1 \cup L_2) \subseteq X_1$, proving the first part of the statement.
For the final part of the statement, note that if $(L_1 \cup L_2 \cup L_3) \subseteq X_1$, then $|X_1| \ge 2k$, so $\sum_{j = 1}^n \left\lfloor \frac{1}{2} |X_j| \right\rfloor \ge k$, contradicting (a).
\end{proof}

By Claim~\ref{fat:2LiXj}, we may assume that $(L_1 \cup L_2) \subseteq X_1$ and $L_3 \not\subseteq X_1$.
Suppose $L_3 \subseteq X_2$ and $N(L_3) \not\subseteq (X_1 \cup X_2 \cup \dots \cup X_n)$.
Then $N(L_3) \cap (Y_2 \setminus X_2) \ne \emptyset$, so by Claim~\ref{fat:Yj=Xj}, $|X_2| \ge k$.
But $|X_1| \ge |L_1| + |L_2| = (k_1 + k_3) + (k_1 + k_2) = k + k_1 \ge k + 1$, so
$\sum_{j = 1}^n \left\lfloor \frac{1}{2} |X_j| \right\rfloor \ge \left\lfloor \frac{1}{2}(k + 1) \right\rfloor + \left\lfloor \frac{1}{2}k \right\rfloor = k$, contradicting (a).
Therefore, either by the above or by Claim~\ref{fat:LiNbrs}, we may assume that at least one $L_i$ has $N(L_i) \subseteq (X_1 \cup \dots \cup X_n)$.
We need to be more precise with this last sum, so let $c \in \{0, 1\}$ such that $c = |L_1 \cup L_2| \mod 2$.
Then
\begin{align*}
\sum_{j = 1}^n \left\lfloor \frac{1}{2} |X_j| \right\rfloor &\ge \frac{1}{2}(|L_1 \cup L_2| - c) + \frac{1}{3}(|L_3| + k - (1 - c)) \\
 &= k + \frac{1}{6}(2k_1 + k_2 + k_3) - \frac{c}{6} - \frac{1}{3}
\end{align*}
Since $k_1, k_2, k_3 \ge 1$, in all cases we have $k + \frac{1}{6}(2k_1 + k_2 + k_3) - \frac{c}{6} - \frac{1}{3} \ge k$, contradicting (a).
This contradiction completes the proof of Theorem~\ref{thm:fattriangle}.
\hfill$\square$

\section{Separating pairs and $3$-planar graphs}

We now introduce some notation we will use in the rest of the proof.
Let $C$ be a cycle in a graph $G$ and $u,v$ be two distinct vertices on $C$.
Given a fixed orientation of $C$, we denote by $C[u,v]$ the subpath of $C$ from $u$ to $v$, and denote $C(u,v]=C[u,v]\setminus\{u\}$, $C[u,v)=C[u,v]\setminus\{v\}$, and $C(u,v)=C[u,v]\setminus\{u,v\}$.
We also use similar notation for paths.
Let $P$ be a path in a graph $G$.
We use $\pend(P)$ to denote the set of endpoints of $P$ and define $\pint(P)=V(P)\setminus \pend(P)$.

The following notion of separating pairs was first studied by McCarty, Wang, and Yu~\cite{MWY18}.

\begin{definition}
\label{separatingpair}
Let $G$ be a graph, let $C$ be a cycle in $G$, let $u_1, u_2 \in V(C)$ be distinct, and let $A \subseteq V(G) \setminus V(C)$.
Then a $(u_1, u_2, C, A)$-separating pair is a set of paths $\{R_1, R_2\}$ such that there exists an orientation of $C$ so that

(i) for $i \in [2]$, $R_i$ is a subpath of $C[u_i, u_{3 - i}]$,

(ii) for $i\in [2]$, $N(A) \cap V(C[u_i,u_{3-i}]) \subseteq V(R_i)$, and

(iii) the graph $G$ has no edge $uv$ such that $u \in \pint(R_1) \cup \pint(R_2)$ and $v \in V(C) \setminus (V(R_1) \cup V(R_2))$.

Furthermore, for all $i, j \in \{1, 2\}$, if $R_i \ne \emptyset$ then we define $r_i^j$ to be the end of $R_i$ closest to $u_j$ on $C[u_i, u_{3 - i}]$. (Thus, $r_i^1 = r_i^2$ if $R_i$ consists of a single vertex.)
\end{definition}

Clearly, a $(u_1, u_2, C, A)$-separating pair exists as the two paths in $C$ between $u_1$ and $u_2$ form such a pair.
For convenience, a $(u_1, u_2, C, A)$-separating pair is said to be {\em special} if $|V(R_1)|+|V(R_2)|$ is minimum.

\begin{lemma}[McCarty, Wang, and Yu~\cite{MWY18}]\label{lem:ConnSep}
Let $\{R_1, R_2\}$ be a special $(u_1, u_2, C, A)$-separating pair in a graph $G$, and let $ x\in \pint(R_1)$.
Then

(i) there exists $a \in V(C)$ with $N(a) \cap A \ne \emptyset$ such that the graph $G[V(C)]$ contains a path through $x, u_1, u_2, a$ in order (possibly $a = u_2$),

(ii) if $G[A]$ is connected then $G[V(C) \cup A] \setminus \pint(R_1)$ contains a cycle through $u_1$ and $u_2$, and

(iii) if $G[A]$ is connected then for every vertex $x' \in \pint(R_2)$, $G[V(C) \cup A]$ contains a path through $x, u_1, u_2, x'$ in order.
\end{lemma}

\medskip
We also need to introduce $3$-planar graphs, which were first used by Seymour~\cite{S80}.
Such graphs were further studied by X. Yu~\cite{Y03}, from which we get the subsequent Definition~\ref{minidefinition}.

\begin{definition}
\label{3-planardefinition}
A $3$-planar graph $(G,\mathcal{A})$ consists of a graph $G$ and a family $\mathcal{A}=\{A_1,\ldots,A_k\}$ of pairwise disjoint subsets of $V(G)$ (allowing $\mathcal{A}=\emptyset$) such that

(i) for distinct $i, j \in \{1, \dots, k\}$, $N(A_i)\cap A_j=\emptyset$,

(ii) for $i \in \{1, \dots, k\}$, $|N(A_i)|\le 3$, and

(iii) if $p(G,\mathcal{A})$ denotes the graph obtained from $G$ by (for each $i$) deleting $A_i$ and adding new edges joining every pair of distinct non-adjacent vertices in $N(A_i)$, then $p(G,\mathcal{A})$ can be drawn in a closed disc $D$ with no pair of edges crossing such that, for each $A_i$ with $|N(A_i)|=3$, $N(A_i)$ induces a facial triangle in $p(G,\mathcal{A})$.
\end{definition}

If, in addition, $b_1,\ldots,b_n$ are some vertices in $G$ such that $b_i\notin A_j$ for any $A_j\in \mathcal{A}$ and $b_1,\ldots,b_n$ occur on the boundary of $D$ in that cyclic order, then we say that $(G,\mathcal{A},b_1,\ldots,b_n)$ is $3$-planar.
We will say that such a drawing is a {\em plane drawing of} $(G,\mathcal{A},b_1,\ldots,b_n)$.
We will say that $(G,b_1,\ldots,b_n)$ is $3$-planar if there exists a collection $\mathcal{A}$ so that $(G,\mathcal{A},b_1,\ldots,b_n)$ is $3$-planar.
If $(G,\emptyset,b_1,\ldots,b_n)$ is $3$-planar we will say that $(G,b_1,\ldots,b_n)$ is {\em planar}.

\begin{definition}
\label{minidefinition}
Let $(G,\mathcal{A})$ be $3$-planar, let $A\in\mathcal{A}$ with $N(A)=\{a_1,\ldots,a_m\}$ (where $m\le3$), and let $H=G[A\cup N(A)]$.
We say that $A$ is {\em minimal} if there is no collection $\mathcal{H}$ of pairwise disjoint subsets of $A$ such that $\mathcal{H}\ne\{A\}$ and $(H,\mathcal{H}, a_1,\ldots,a_m)$ is $3$-planar.
We say that $\mathcal{A}$ is minimal if every member of $\mathcal{A}$ is minimal.
\end{definition}

Seymour~\cite{S80} gave the following important characterization of $2$-linked graphs, a result which was also proved independently by Thomassen~\cite{T80}.
McCarty, Wang, and Yu~\cite{MWY18} provide us with two additional useful results.

\begin{theorem}[Seymour~\cite{S80}]\label{linkage}
Let $G$ be a graph and let $s_1, t_1, s_2, t_2$ be distinct vertices of $G$.
Then $G$ contains no $(\{s_1, t_1\}, \{s_2, t_2\})$-linkage if and only if $(G, s_1, s_2, t_1, t_2)$ is $3$-planar.
\end{theorem}

\begin{lemma}[McCarty, Wang, and Yu~\cite{MWY18}]\label{3-planar}
Let $(G,\mathcal{A})$ be 3-planar so that $G$ is 3-connected and $\mathcal{A}$ is minimal. Fix some plane drawing of $p(G,\mathcal{A})$ and let $F$ be the set of vertices on the outer face. Suppose that $|F|\ge4$. Let $s_1,t_1,s_2$ be three distinct vertices in $F$, and let $t_2\in V(G)\setminus F$. Then $G$ has an $(\{s_1,t_1\},\{s_2,t_2\})$-linkage.
\end{lemma}

\begin{lemma}[McCarty, Wang, and Yu~\cite{MWY18}]\label{discharge}
Let $H$ be a $3$-connected planar graph with some fixed planar drawing. Let $Z$ be the outer cycle of $H$ and let $x$ and $y$ be distinct vertices in $V(Z)$. Then either

(i) there exists $v\in V(H)\setminus V(Z)$ with $d(v)\le6$, or

(ii) there exists $uv\in E(Z)$ so that $\{u,v\}\cap \{x,y\}=\emptyset$ and $d(u)+d(v)\le7$.
\end{lemma}

\section{Flowers and their properties}

\begin{definition}\label{flower}
Let $\{u_1, u_2, u_3, u_4\} \subseteq V(G)$.
Then a $(u_1, u_2, u_3, u_4)$-flower is an ordered list $F = (C_1\com C_2\com C_3\com P_1\com P_2\com P_3)$ such that

(i) $C_1, C_2, C_3$ are cycles such that $u_1 \in V(C_1)$, $u_3 \in V(C_2)$, $V(C_1) \cap V(C_2) = \{u_2\}$, and $C_3$ is disjoint from $C_1$ and $C_2$,

(ii) for each $i \in \{1, 2, 3\}$, $P_i$ is a path in $G$ from $u_i$ to some vertex $v_i \in V(C_3)$, and $P_i$ is internally disjoint from $V(C_1) \cup V(C_2) \cup V(C_3)$, and

(iii) $P_1, P_2, P_3$ are pairwise vertex disjoint, and $v_1, v_2, v_3, u_4$ occur on $C_3$ in order.
\end{definition}

An illustration of a flower is given in Figure~\ref{flower}.

\begin{figure}[H]
\includegraphics[scale=0.65]{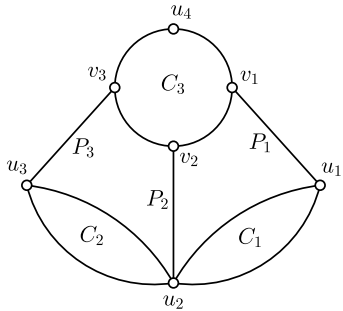}
\caption{A $(u_1,u_2,u_3,u_4)$-flower.}
\label{flower}
\end{figure}

\begin{lemma}\label{lem:flower}
Suppose $G$ is a $7$-connected graph such that for some four vertices $u_2\com u_1\com u_3\com u_4 \in V(G)$, $G$ contains no walk through $u_2, u_1, u_3, u_2, u_4$ in order.
Further suppose that there exist internally disjoint paths $Q_1, \dots, Q_7$ in $G$ avoiding $u_4$ such that $Q_1$ has ends $u_1$ and $u_3$; $Q_2, Q_3, Q_4$ have ends $u_1$ and $u_2$; and $Q_5, Q_6, Q_7$ have ends $u_2$ and $u_3$.
Then, up to symmetry between $u_1$ and $u_3$, $G$ contains a $(u_1, u_2, u_3, u_4)$-flower.
\end{lemma}

\begin{proof}
Since $G$ is $7$-connected, there exist $5$ internally disjoint paths $P_1, \ldots, P_5$ from $u_4$ to $u_2$ which avoid $u_1$ and $u_3$.
Throughout the remainder of this proof, we will assume the paths $P_i$ are read starting at $u_4$.

For each $i \in [5]$, we may assume that the first vertex of $P_i$ meeting $\cup_{j = 1}^7 Q_j$ belongs to $\pint(Q_1)$.
For otherwise, the first vertex $q$ of some $P_i$ meeting $\cup_{j = 1}^7 Q_j$ belongs to $V(Q_k) \setminus \{u_1, u_3\}$ for some $k \in \{2, \dots, 7\}$, say $k = 2$.
Then $Q_3 \cup Q_1 \cup Q_5 \cup Q_2[u_2, q] \cup P_i[q, u_4]$ is a kite-linkage of $u_2\com u_1\com u_3\com u_4$, a contradiction.

Thus for $i \in [5]$, let $w_i$ be the first vertex of $P_i$ meeting $Q_1$.
By relabelling if necessary, we may assume that $w_1, \ldots, w_5$ occur on $Q_1(u_1, u_3)$ in order.
For $i \in [5]$, additionally let $w_i''$ be the first vertex of $P_i$ meeting $\cup_{j = 2}^7 Q_j$ and $w_i'$ the last vertex of $P_i[u_4, w_i'']$ meeting $Q_1[w_1, w_5]$.

Let $u_i'$ be the vertex of $(\cup_{j = 1}^5 P_j[u_4, w_j'']) \cap \pint(Q_1)$ closest to $u_i$ on $Q_1$ for each $i \in \{1, 3\}$.
If $u_1', u_3'$ are on the same path $P_i$ for some $i \in [5]$, then we can find a path $P = Q_1[u_1, u_1'] \cup P_i[u_1', u_3'] \cup Q_1[u_3', u_3]$ from $u_1$ to $u_3$.   Thus a kite-linkage of $u_2\com u_1\com u_3\com u_4$ can easily be found by using $P, P_j, \cup_{k = 2}^7 Q_k$, where $i \ne j$, a contradiction.

So we may assume that $u_1' \in V(P_i), u_3' \in V(P_j)$ with $i \ne j$.
For each $k \in \{i, j\}$, let $u_k''$ be the vertex such that $V(P_k[u_k'',u_k'])\cap V(Q_1[w_1, w_5]) = \{u_k''\}$.
Now choose a path $P_k$ such that $k \notin \{i, j\}$.  Let $C_3$ be the cycle formed by $P_1[u_4,w_1] \cup Q_1[w_1,w_5] \cup P_5[w_5,u_4]$, and orient $C_3$ so that $u_4, w_1, w_5$ occur in order.
By symmetry between $u_1$ and $u_3$, either $u_1'' \in V(C_3(w_k', w_5])$ or $u_1'' \in V(C_3[w_1, w_k'))$ and $u_3'' \in V(C_3(w_k', w_5])$.
In the former case, we can find a path $P = C_3(u_4, w_k'] \cup P_k[w_k', u_2]$ from $u_4$ to $u_2$ and a path $P' = Q_1[u_1, u_1'] \cup P_i[u_1', u_1''] \cup Q_1[u_1'',u_3]$ from $u_1$ to $u_3$.
Then we can find a kite-linkage of $u_2\com u_1\com u_3\com u_4$ by using $P, P', \cup_{\ell = 2}^7 Q_\ell$, a contradiction.
In the latter case, by symmetry we may assume that $w_k''$ belongs to $Q_2$.
Then let $C_1 = Q_5 \cup Q_6$, $C_2 = Q_3 \cup Q_4$, $P_1 = Q_1[u_1, u_1'] \cup P_i[u_1', u_1'']$, $P_3 = Q_1[u_3, u_3'] \cup P_j[u_3', u_3'']$, and $P_2 = Q_2[u_2, w_k''] \cup P_k[w_k'', w_k']$.
Then $F = (C_1\com C_2\com C_3\com P_1\com P_2\com P_3)$ is the desired flower.
\end{proof}

\begin{theorem}\label{flowerthm}
Suppose $G$ is an $8$-connected graph such that for some four vertices $u_2\com u_1\com u_3\com u_4 \in V(G)$, $G$ contains no walk through $u_2, u_1, u_3, u_2, u_4$ in order.
Then, up to symmetry between $u_1$ and $u_3$, $G$ contains a $(u_1, u_2, u_3, u_4)$-flower.
\end{theorem}

\begin{proof}
Since $G$ is $8$-connected, $G - u_4$ is $7$-connected.
By Theorem~\ref{thm:fattriangle}, $G$ is $F_{3, 3, 1}$-linked.
Thus there exist internally disjoint paths $Q_1, \dots, Q_7$ such that $Q_1$ has ends $u_1$ and $u_3$; $Q_2, Q_3, Q_4$ have ends $u_1$ and $u_2$; and $Q_5, Q_6, Q_7$ have ends $u_2$ and $u_3$.
Therefore, by Lemma~\ref{lem:flower} $G$ contains a $(u_1, u_2, u_3, u_4)$-flower.
\end{proof}


By proceeding along similar lines as in~\cite{MWY18}, we can prove the following properties of the flower.

\begin{proposition}\label{prop:Conn}
Suppose $G$ is a $7$-connected graph such that for some four vertices $u_2\com u_1\com u_3\com u_4 \in V(G)$, $G$ contains no walk through $u_2, u_1, u_3, u_2, u_4$ in order.
Further suppose that $G$ has a $(u_1, u_2, u_3, u_4)$-flower.
Then $G$ has a $(u_1, u_2, u_3, u_4)$-flower $F = (C_1\com C_2\com C_3\com P_1\com P_2\com P_3)$ such that

(i) for each $i \in [2]$, $G[V(C_i)]$ has no cycle through $u_i$ and $u_{i + 1}$ with fewer vertices than $C_i$,

(ii) for every $j \in [3]$, $|V(P_j)| = 2$, and

(iii) the graph $H = G - (V(C_1) \cup V(C_2))$ is $3$-connected.
\end{proposition}

\noindent{\em Proof.}
Let $F = (C_1\com C_2\com C_3\com P_1\com P_2\com P_3)$ be a flower in $G$.
Let $B$ be the block of $H$ containing $C_3$.
Let $B_1, B_2, \dots, B_m$ be the components of $H - V(B)$ with non-empty intersection with $V(F)$, and let $A_1, A_2, \dots, A_n$ be the components of $H - V(B)$ with empty intersection with $V(F)$.
We may assume that $F$ is chosen so that

(1) $\sum_{i = 1}^3 |V(P_i)|$ is minimum,

(2) subject to (1), $|V(B)|$ is maximum,

(3) subject to (2), $(|V(B_1)|, \dots, |V(B_m)|)$ is maximal with respect to lexicographic ordering, and

(4) subject to (3), $(|V(A_1)|, \dots, |V(A_n)|)$ is maximal with respect to lexicographic ordering.

Let $C_1$ and $C_2$ have arbitrary fixed orientations, and fix an orientation of $C_3$ so that $v_1, v_2, v_3, u_4$ occur in order.
We now prove the statement via a series of claims.

\begin{claim}\label{edge}
For each $i \in [2]$, $N(C_i \setminus \{u_2\}) \cap V(C_{3 - i}) =\emptyset$.
\end{claim}
\begin{proof}
Suppose otherwise.
Then, by symmetry, there is an edge in $G$ with one end $u \in V(C_1(u_2, u_1])$ and the other end $v \in V(C_2[u_3, u_2))$.
Then $C_1[u, u_2] \cup uv \cup C_2[u_2, v] \cup P_2 \cup C_3[v_2, u_4]$ forms a kite-linkage of $u_2\com u_1\com u_3\com u_4$, a contradiction.
\end{proof}

\begin{claim}\label{smaller}
For each $i \in [2]$, the graph $G[V(C_i)]$ has no cycle through $u_i$ and $u_{i + 1}$ with fewer vertices than $C_i$.
\end{claim}
\begin{proof}
Suppose to the contrary that, say, $C_1'$ is a cycle in $G[V(C_1)]$ through $u_1$ and $u_2$ such that $|V(C_1')| < |V(C_1)|$.
Then $F' = (C_1'\com C_2\com C_3\com P_1\com P_2\com P_3)$ is a flower in $G$.
As no path $P_j$ has changed for any $j \in [3]$, (1) holds for $F'$.
Let $B'$ be the block of $G - (V(C_1') \cup V(C_2))$ containing $C_3$.
Let $D$ be a component of $G[V(C_1)]\setminus V(C_1')$.
Since $V(B) \subseteq V(B')$, by (2) we have $V(B) = V(B')$, so (2) holds for $F'$.
If $V(D) \cap V(F') \ne \emptyset$, or if $V(D) \cap V(B_i) \ne \emptyset$ for some $i \in [n]$, then $F'$ contradicts the choice of $F$ by (3).
But then $V(D) \cap V(F') = \emptyset$, so $F'$ contradicts the choice of $F$ by (4).
\end{proof}

\begin{claim}\label{empty}
There is no component of $H - V(B)$ with empty intersection with $V(F)$.
\end{claim}

\begin{proof}
Suppose otherwise.
Then $A_n$ exists.
By definition, $|N(A_n) \cap V(B)| \le 1$.
Since $G$ is $7$-connected, $A_n$ must have at least three neighbors on $(V(C_1) \cup V(C_2)) \setminus\{u_1, u_2, u_3\}$.
By symmetry, we may assume $A_n$ has at least two neighbors on $V(C_1) \setminus\{u_1, u_2\}$.
Let $\{R_1, R_2\}$ be a special $(u_1, u_2, C_1, A_n)$-separating pair in $G$.
Let $X = \pint(R_1) \cup \pint(R_2) \cup V(A_n)$, $T = \pend(R_1) \cup \pend(R_2) \cup (N(A_n) \cap V(H))$, and $Y = V(G) \setminus (X \cup T)$.
Since $|T| \le 5$ and $G$ is $7$-connected, there exist $x, x' \in X$ and $y, y' \in Y$ such that $xy, x'y' \in E(G)$.
We claim that $x, x' \in \pint(R_1) \cup \pint(R_2)$ and $y, y' \in V(H) \setminus (X \cup T)$.
For otherwise, by Defintion~\ref{separatingpair}, Claim~\ref{edge}, and the symmetry of $xy,x'y'$, we may assume that $x\in V(A_n), y\in V(C_2(u_2, u_3])$.
By the symmetry of $R_1$ and $R_2$, let $a' \in V(R_1)$ be a neighbor of $a \in V(A_n)$ and let $P$ be a path in $A_n$ with ends $a$ and $x$.
Then $C_1[u_2,a']\cup a'a\cup P\cup xy\cup C_2[y,u_2]\cup P_2\cup C_3[v_2,u_4]$ forms a kite-linkage of $u_2\com u_1\com u_3\com u_4$, a contradiction.
This completes the claim.

By the symmetry of $xy, x'y'$ and $R_1, R_2$, either $x \in \pint(R_1)$ and $y \in V(H) \setminus (X \cup T \cup V(B))$; or $x,x' \in \pint(R_1)$ and $y,y'\in V(B)$; or $x \in \pint(R_1)$, $x' \in \pint(R_2)$, and $y, y' \in V(B)$.
In the first two cases, by Lemma~\ref{lem:ConnSep}(ii) there exists a cycle $C_1'$ containing $u_1, u_2$ in $G[V(C_1) \cup V(A_n)] - \pint(R_1)$.
Then $F' = (C_1'\com C_2\com C_3\com P_1\com P_2\com P_3)$ is a flower in $G$.
Let $B'$ be the block of $G - (V(C_1) \cup V(C_2))$ which contains $C_3$.
Since no path $P_i$ for any $i \in [3]$ has been altered, (1) still holds for $F'$.
If $y, y' \in V(B)$, then $V(B) \cup \pint(R_1) \subseteq V(B')$, and so $|V(B')| > |V(B)|$, contradicting (2).
Otherwise $y \notin V(B)$, so whichever component $B_1, \dots, B_m, A_1, \dots, A_{n - 1}$ contains $y$ grows by including $\pint(R_1)$, contradicting either (3) or (4).
In the third case, since $B$ is $2$-connected, and by the symmetry of $y, y'$, there exist two disjoint paths $P, P'$ in $B$ with ends $y, v_3$ and $y', u_4$, respectively.
By Lemma~\ref{lem:ConnSep}(iii), $G[V(C_1) \cup V(A_n)]$ contains a path $Q$ through $x, u_1, u_2, x'$ in order.
Now $P' \cup x'y' \cup Q \cup xy \cup P \cup P_3 \cup C_3[u_3, u_2]$ forms a kite-linkage of $u_2\com u_1\com u_3\com u_4$, a contradiction.
\end{proof}

\begin{claim}\label{P1P3}
$|V(P_1)|=|V(P_2)|=|V(P_3)| = 2$.
\end{claim}
\begin{proof}
By the symmetry between $P_1$ and $P_3$, we suppose to the contrary that $|V(P_i)| > 2$ for some $i \in [2]$.
Let $X_1$ be the component of $H - v_i$ containing $\pint(P_i)$, and let $X_2$ be the component of $H - v_i$ containing $v_{i+1}$.
We may assume that $H - v_i = X_1 \cup X_2$, for otherwise, there is a component of $H - v_i$ with empty intersection with $V(F)$, contrary to Claim~\ref{empty}.

First suppose that $X_1 = X_2$.
Then $H - v_i$ is connected, so there exists a shortest path $P$ in $X_1$ with one end $x\in\pint(P_i)$ and the other end $y\in(V(F) \cap V(H)) \setminus V(P_i)$.
So if $i = 1$, then either $y \in V(P_2) \cup V(C_3[u_4, v_2])$ or $y \in V(P_3) \cup V(C_3(v_2,u_4))$, and if $i = 2$, then either $y \in V(P_1) \cup V(P_3) \cup V(C_3[v_1, v_3])$ or $y \in V(C_3(v_3, v_1))$.
Suppose $i = 1$ and $y \in V(P_2)$.
Then let $P_1' = P_1[u_1, x]$, $P_2' = P_2[u_2, y]$, $C_3' = C_3[v_2, v_1] \cup P_1 [v_1, x] \cup P \cup P_2[y, v_2]$, and $F' = (C_1\com C_2\com C_3'\com P_1'\com P_2'\com P_3)$.
Then $F'$ is a flower contradicting the choice of $F$ by (1).
The case $y \in V(C_3[u_4, v_2])$, and the cases $y \in V(P_1) \cup V(P_3) \cup V(C_3[v_1, v_3])$ when $i = 2$ are similar.
So suppose $i = 1$ and $y \in V(P_3) \cup V(C_3(v_2,u_4))$.
Then a $\{(u_1, u_3), (u_2, u_4)\}$-linkage can be easily found in $C_3 \cup P_1 \cup P_2 \cup P_3 \cup P$, giving a kite-linkage of $u_2\com u_1\com u_3\com u_4$ by adding $C_1[u_1, u_2]$ and $C_2[u_2, u_3]$, a contradiction.
The case $y \in V(C_3(v_3, v_1))$ when $i = 2$ is similar.

Thus $X_1 \ne X_2$.
Suppose $|V(P_1)| > 2$.
Then we have $N(X_1)\cap (V(C_2)\setminus \{u_2\})=\emptyset$.
For otherwise, by symmetry let $aa'\in E(G)$ such that $a\in X_1, a'\in V(C_2(u_2, u_3])$, and let $P$ be a path in $X_1 \cup \{u_1\}$ between $a$ and $u_1$.
Then $C_1[u_2, u_1] \cup P \cup aa' \cup C_2[a', u_2] \cup P_2 \cup C_3[v_2, u_4]$ forms a kite-linkage of $u_2\com u_1\com u_3\com u_4$, a contradiction.
Thus since $G$ is $7$-connected and $X_1$ has exactly one neighbor $v_1$ in $H$, $|N(X_1) \cap V(C_1)| \ge 6$.
Let $\{R_1, R_2\}$ be a special $(u_1, u_2, C_1, X_1)$-separating pair in $G$.
Note that $u_1$ has a neighbor in $X_1$, so $u_1 \in \pend(R_1) \cap \pend(R_2)$.
Let $X = \pint(R_1) \cup \pint(R_2) \cup V(X_1)$, $T = \pend(R_1) \cup \pend(R_2) \cup \{v_1\}$, and $Y = V(G) \setminus (X \cup T)$.
Thus $|T| \le 4$, so since $G$ is $7$-connected, there exist $x \in X$ and $y \in Y$ such that $xy \in E(G)$.
Since $X_1$ and $X_2$ are components of $H - v_1$, we may assume by the definition of separating pair and symmetry that $x \in \pint(R_1)$ and $y \in V(X_2)$.
Now there exists a shortest path $P$ in $X_2$ from $y$ to a vertex $y' \in V(X_2) \cap V(F)$ (possibly $y = y'$).
If $y' \in V(C_3[u_4, v_1))$, then $C_3[u_4, y'] \cup P \cup yx \cup C_1[x, u_1] \cup P_1 \cup C_3[v_1, v_3] \cup P_3 \cup C_2[u_3, u_2]$ forms a kite-linkage of $u_2\com u_1\com u_3\com u_4$, a contradiction.
Otherwise, there exists a path $P'$ in $C_3(v_1, u_4) \cup P_2 \cup P_3$ from $y'$ to $u_3$.
By Lemma~\ref{lem:ConnSep}(i) there exist vertices $a \in V(C_1), a' \in V(X_1)$ such $aa' \in E(G)$ and $G[V(C_1)]$ contains a path $Q$ through $x, u_1, u_2, a$ in order.
Let $Q'$ be a path in $X_1 \cup P_1$ with ends $a'$ and $v_1$.
Then $C_3[u_4, v_1] \cup Q' \cup aa' \cup Q \cup xy \cup P \cup P' \cup C_2[u_3, u_2]$ forms a kite-linkage of $u_2\com u_1\com u_3\com u_4$, a contradiction.

Hence we may assume $|V(P_2)| > 2$ and $\pint(P_2) \subseteq V(X_1)$.
Since $G$ is $7$-connected and $X_1$ has exactly one neighbor $v_2$ in $H$, $|N(X_1) \cap (V(C_1) \cup V(C_2))| \ge 6$.
Let $\{R_1, R_2\}$ be a special $(u_1, u_2, C_1, X_1)$-separating pair and $\{R_3, R_4\}$ be a special $(u_2, u_3, C_2, X_1)$-separating pair in $G$.
Note that $u_2$ has a neighbor in $X_1$, so $u_2$ is an end of $R_i$ for each $i \in [4]$.
Let $X = (\cup_{i = 1}^4 \pint(R_i)) \cup V(X_1)$, $T = (\cup_{i = 1}^4 \pend(R_i)) \cup \{v_2\}$, and $Y = V(G) \setminus (X \cup T)$.
Then $|T| \le 6$, so because $G$ is $7$-connected, there exist $x \in X$ and $y \in Y$ such that $xy \in E(G)$.
By the definition of separating pair and the symmetry of $R_1, \dots, R_4$, we may assume that $x \in \pint(R_1)$ and $y \in V(X_2)$.
Then there exists a shortest path $P$ in $X_2$ from $y$ to a vertex $y' \in V(X_2) \cap V(F)$.
If $y' \in V(C_3(v_2, v_3])\cup V(P_3)$, then $C_3(v_2, v_3] \cup P_3 \cup P$ contains a path $P'$ from $y$ to $u_3$.
Then $C_3[u_4, v_2] \cup P_2 \cup C_1[u_2, x] \cup xy \cup P' \cup C_2[u_3, u_2]$ forms a kite-linkage of $u_2\com u_1\com u_3\com u_4$, a contradiction.
Otherwise, $y' \in V(C_3(v_3, v_2)) \cup V(P_1)$, so there exists a path $P'$ from $y$ to $u_4$ in $C_3(v_3, v_2) \cup P_1 \cup P$.
By Lemma~\ref{lem:ConnSep}(i) there exist vertices $a \in V(C_1)$ and $a' \in V(X_1)$ such that $aa' \in E(G)$ and $G[V(C_1)]$ contains a path $Q$ through $x, u_2, u_1, a$ in order.
Let $Q'$ be a path in $X_1$ with ends $a'$ and $v_2$.
Then $P' \cup xy \cup Q \cup aa' \cup Q' \cup C_3[v_2, v_3] \cup P_3 \cup C_2[u_2, u_3]$ forms a kite-linkage of $u_2\com u_1\com u_3\com u_4$, a contradiction.
\end{proof}

By Claims~\ref{smaller} and \ref{P1P3}, (i) and (ii) hold.
Furthermore, $H$ is $2$-connected, since $V(F) \cap V(H) \subseteq V(B)$ by Claim~\ref{P1P3}, and so $V(H) = V(B)$ by Claim~\ref{empty}.
It only remains to prove (iii), which we accomplish with the next claim.

\begin{claim}\label{3conn}
$H$ is $3$-connected.
\end{claim}
\begin{proof}
Suppose otherwise.
By the above, $H$ is $2$-connected, so let $T' = \{t_1, t_2\}$ be a 2-cut in $H$.
We may assume that $T'$ and a component $A$ of $H - T'$ are chosen so that 
$|V(A) \cap \{v_1, v_2, v_3, u_4\}|$ is minimum. 
Then $|V(A) \cap \{v_1, v_2, v_3, u_4\}| \le 2$, and $|N(A) \cap (V(C_1) \cup V(C_2))| \ge 5$ since $G$ is $7$-connected.
By the symmetry of $v_1$ and $v_3$, we consider the following three cases.

Case 1: $|V(A) \cap \{v_1, v_2, v_3, u_4\}| = 2$.
By the choice of $A$, $H - T' = A \cup A'$, where $A'$ is a component of $H - T'$ which also contains two vertices of $\{v_1, v_2, v_3, u_4\}$, so by symmetry we may assume $u_4 \notin V(A)$.
Then either $\{v_1, v_3\} \subseteq V(A)$ or $\{v_2, v_3\} \subseteq V(A)$.
Suppose first that $\{v_1, v_3\} \subseteq V(A)$.
Note that $H - A$ is also connected.
Let $P$ be a path in $A$ between $v_1$ and $v_3$, and let $Q$ be a path in $H - A$ between $v_2$ and $u_4$.
Then $C_1[u_2, u_1] \cup P_1 \cup P \cup P_3 \cup C_2[u_3, u_2] \cup P_2 \cup Q$ forms a kite-linkage of $u_2\com u_1\com u_3\com u_4$, a contradiction.

Thus $\{v_2, v_3\} \subseteq V(A)$, and $\{v_1, u_4\} \subseteq V(A')$.
Since $C_3$ is $2$-connected, we may assume that $t_1 \in C_3(v_1, v_2), t_2 \in C_3(v_3, u_4)$.
Recall that $|N(A) \cap (V(C_1) \cup V(C_2))| \ge 5$.
Let $\{R_1, R_2\}$ be a special $(u_1, u_2, C_1, A)$-separating pair, and let $\{R_3, R_4\}$ be a special $(u_2, u_3, C_2, A)$-separating pair in $G$.
Since $u_2$ and $u_3$ each have a neighbor in $A$, $u_2$ is an end of $R_i$ for $i \in \{1, \dots, 4\}$, and $u_3$ is an end of $R_i$ for $i \in \{3, 4\}$.
Let $X = (\cup_{i = 1}^4 \pint(R_i)) \cup V(A)$, $T = (\cup_{i = 1}^4 \pend(R_i)) \cup T'$, and $Y = V(G) \setminus (X \cup T)$.
Then $|T| \le 6$.
Since $G$ is $7$-connected, there exist $x \in X$ and $y \in Y$ such that $xy \in E(G)$.
By the definition of separating pair and symmetry, we may assume that $x \in \pint(R_1) \cup \pint(R_3)$ and $y \in V(A')$.
If $x \in \pint(R_1)$, then by Lemma~\ref{lem:ConnSep}(i) there exist two vertices $a' \in V(C_1), a \in V(A)$ such $aa' \in E(G)$ and $G[V(C_1)]$ contains a path $P$ through $x, u_2, u_1, a'$ in order.
Let $Q$ be a path in $A$ between $a$ and $v_3$ and let $Q'$ be a path in $A'$ between $y$ and $u_4$.
Then $Q' \cup xy \cup P \cup a'a \cup Q \cup P_3 \cup C_2[u_3, u_2]$ forms a kite-linkage of $u_2\com u_1\com u_3\com u_4$, a contradiction.
Otherwise, if $x \in \pint(R_3)$, then let $P$ be a shortest path in $A'$ from $y$ to a vertex $y' \in C_3(t_2, t_1)$.
If $y' \in C_3(t_2, v_1)$, let $P'$ be the subpath of $C_3(t_2, v_1)$ with ends $y', u_4$.
Then $P' \cup P \cup xy \cup C_2[x, u_2] \cup C_1[u_2, u_1] \cup P_1 \cup C_3[v_1, v_3] \cup P_3$ forms a kite-linkage of $u_2\com u_1\com u_3\com u_4$.
If $y' \in C_3[v_1, t_1)$, then $C_1[u_2, u_1] \cup P_1 \cup C_3[v_1, y'] \cup P \cup xy \cup C_2[x, u_2] \cup P_2 \cup C_3[v_2, u_4]$ forms a kite-linkage of $u_2\com u_1\com u_3\com u_4$, a contradiction.

Case 2: $V(A) \cap \{v_1, v_2, v_3, u_4\} = \{u_4\}$.
Since $|N(A) \cap (V(C_1) \cup V(C_2))| \ge 5$, by symmetry, we may assume that there exist $a' \in V(C_1(u_1, u_2)), a \in V(A)$ such that $aa' \in E(G)$.
Let $P$ be a path in $A$ between $u_4$ and $a$, and let $Q$ be a path in $H - A$ between $v_1$ and $v_3$.
Then $P \cup aa' \cup C_1[u_2, a'] \cup C_2[u_2, u_3] \cup P_3 \cup Q \cup P_1$ forms a kite-linkage of $u_2\com u_1\com u_3\com u_4$, a contradiction.

Case 3: $|V(A) \cap \{v_1, v_2, v_3, u_4\}| \le 1$ and $v_3, u_4 \notin V(A)$.
If $v_2 \in V(A)$, then since $C_3$ is $2$-connected we may assume $t_1 \in V(C_3(v_1, v_2)), t_2 \in V(C_3(v_2, v_3))$.
Let $\{R_1, R_2\}$ be a special $(u_1, u_2, C_1, A)$-separating pair and let $\{R_3, R_4\}$ be a special $(u_2, u_3, C_2, A)$-separating pair in $G$.
Note that $u_2$ is an end of $R_i$ for each $i \in [4]$.
If $V(R_i) = \{u_2\}$ for some $i \in [4]$, then let $X = (\cup_{i = 1}^4 \pint(R_i) ) \cup V(A)$, and otherwise let $X = (\cup_{i = 1}^4 \pint(R_i) ) \cup V(A) \cup \{u_2\}$.
Also let $T = (\cup_{i = 1}^4 \pend(R_i)) \cup T' \setminus\{u_2\}$, and $Y = V(G) \setminus (X \cup T)$.
Then in either case, $|T| \le 6$.
Since $G$ is $7$-connected, there exist $x \in X$ and $y \in Y$ such that $xy \in E(G)$.
Suppose $x = u_2$.
Then by Claim~\ref{smaller}, $y \in V(H) \setminus (V(A) \cup T')$.
Let $P$ be a shortest path in $H - A$ from $y$ to some vertex $y' \in C_3$.
If $y' \in V(C_3(v_3, v_1))$, let $P'$ be the subpath of $C_3(v_3, v_1)$ with ends $y', u_4$.
Then $P' \cup P \cup yu_2 \cup C_2[u_2, u_3] \cup P_3 \cup C_3[v_1, v_3] \cup P_1 \cup C_1[u_1, u_2]$ forms a kite-linkage of $u_2\com u_1\com u_3\com u_4$, a contradiction.
By the symmetry between $C_3[v_1, t_1)$ and $C_3(t_2, v_3]$, we may thus assume $y' \in V(C_3[v_1, t_1))$.
Since $u_2 \in X$, $V(R_i) \ne \{u_2\}$ for $i \in \{1, 2\}$.
Since $\{R_1, R_2\}$ is special, there exists $a' \in V(R_1) \cup V(R_2) \setminus \{u_2\}$ and $a \in V(A)$ such that $aa' \in E(G)$, say $a' \in V(R_1)$.
Let $Q$ be a path in $A \cup T'$ with ends $a, t_2$.
Then $C_3[u_4, y'] \cup P \cup C_1[a', u_1] \cup aa' \cup Q \cup C_3[t_2, v_3] \cup P_3 \cup C_2[u_2, u_3]$ forms a kite-linkage of $u_2\com u_1\com u_3\com u_4$, a contradiction.
Hence $x \ne u_2$, so by the definition of separating pair and symmetry of $R_1, \dots, R_4$, we assume $x \in \pint(R_1)$ and $y \in V(H) \setminus (V(A) \cup T')$. 

If $v_2 \notin V(A)$, then we will show that $A$ has no neighbors in one of $C_1 - u_2$ or $C_2 - u_2$.
Suppose otherwise that $a \in V(A)$ has a neighbor $a' \in V(C_1(u_2, u_1])$ and $b \in V(A)$ has a neighbor $b' \in V(C_2(u_2, u_3])$.
Let $P$ be a path in $A$ between $a$ and $b$, and let $Q$ be a path in $H - A$ between $v_2$ and $u_4$.
Then $C_1[u_1, a'] \cup a'a \cup P \cup bb' \cup C_2[b', u_2] \cup P_2 \cup Q$ forms a kite-linkage of $u_2\com u_1\com u_3\com u_4$, a contradiction.
So we may assume that $N(A) \cap (V(C_2) \setminus \{u_2\}) = \emptyset$.
Let $\{R_1, R_2\}$ be a special $(u_1, u_2, C_1, A)$-separating pair in $G$.
Let $X = \pint(R_1) \cup \pint(R_2) \cup V(A)$, $T = \pend(R_1) \cup \pend(R_2) \cup T'$, and $Y = V(G) \setminus (X \cup T)$.
Since $|T| \le 6$ and $G$ is $7$-connected, there exist $x \in X$ and $y \in Y$ such that $xy \in E(G)$.
By the definition of separating pair and symmetry of $R_1$ and $R_2$, we assume $x \in \pint(R_1)$ and $y \in V(H) \setminus (V(A) \cup T')$.  

Therefore, in any case, we may assume there exists an edge $xy$ with $x \in \pint(R_1)$ and $y \in V(H) \setminus (V(A) \cup T')$.

Now we show that there exists a vertex $t \in T'$, such that $H - A$ contains two disjoint paths $Q_3, Q_4$ from $\{t, y\}$ to $v_3, u_4$, respectively.
Otherwise, there is a separation $(M, N)$ of $H - A$ of order at most $1$ so that $T' \cup \{y\} \subseteq M$ and $\{v_3, u_4\} \subseteq N$.
But then $(M \cup A, N)$ is a separation of $H$ with order at most 1, contrary to the fact that $H$ is 2-connected.
Observe that $\{R_1, R_2\}$ is also a special $(u_2, u_1, C_1, A)$-separating pair if we reverse the orientation of $C_1$.
Thus by Lemma~\ref{lem:ConnSep}(i), for $i \in [2]$, there is an edge $a_i a_i' \in E(G)$ so that $a_i \in V(A), a_i' \in V(C_1)$, and $G[V(C_1)]$ contains a path $Q_i$ with ends $x$ and $a_i'$ going through $x, u_{3 - i}, u_i, a_i'$ in order.
If $t \in V(Q_3), y \in V(Q_4)$, then let $P$ be a path in $A \cup T$ between $t$ and $a_1$.
Now $Q_4 \cup xy \cup Q_1 \cup a_1' a_1 \cup P \cup Q_3 \cup P_3 \cup C_2[u_3, u_2]$ forms a kite-linkage of $u_2\com u_1\com u_3\com u_4$, a contradiction.
Otherwise, $t \in V(Q_4)$ and $y \in V(Q_3)$, so let $P$ be a path in $A \cup T$ between $t$ and $a_2$.
Now $Q_4 \cup P \cup a_2 a_2' \cup Q_2 \cup xy \cup Q_3 \cup P_3 \cup C_2[u_3, u_2]$ forms a kite-linkage of $u_2\com u_1\com u_3\com u_4$, a contradiction.
\end{proof}

\section{Proof of Theorem~\ref{kite}}\label{sec:final}

Before we can complete the proof of Theorem~\ref{kite}, we need two final results.

\begin{lemma}\label{prop}
Suppose $G$ is a $7$-connected graph such that for some four vertices $u_2\com u_1\com u_3\com u_4 \in V(G)$, $G$ has a $(u_1\com u_2\com u_3\com u_4)$-flower $F = (C_1\com C_2\com C_3\com P_1\com P_2\com P_3)$ and $G$ contains no walk through $u_2, u_1, u_3, u_2, u_4$ in order.
Then $N(C_i - u_2) \subseteq V(C_3[v_i, v_{i + 1}])$ for each $i \in [2]$.
\end{lemma}
\begin{proof}
Suppose otherwise.
We may assume that $F$ is a $(u_1\com u_2\com u_3\com u_4)$-flower as in Proposition~\ref{prop:Conn}, and thus $H$ is $3$-connected.
If $H$ contains a $(\{v_1, v_3\}, \{v_2, u_4\})$-linkage, then a kite-linkage of $u_2\com u_1\com u_3\com u_4$ can easily be found, a contradiction.
Thus by Theorem~\ref{linkage}, $(H, v_1, v_2, v_3, u_4)$ is $3$-planar.
By symmetry, suppose first that there exists $u \in V(C_1(u_2, u_1])$ and $v \in V(H) \setminus V(C_3[v_1, v_2])$ such that $uv \in E(G)$.
Suppose $v \in V(H) \setminus V(C_3)$.
By Lemma~\ref{3-planar}, $H$ contains disjoint paths $P, Q$ with ends $v, v_3$ and $v_2, u_4$, respectively.
Then $Q \cup P_2 \cup C_1[u_1, u] \cup uv \cup P \cup P_3 \cup C_2[u_3, u_2]$ forms a kite-linkage of $u_2\com u_1\com u_3\com u_4$, a contradiction.
Thus $v \in V(C_3(v_2, v_1))$.
If $v \in V(C_3(v_2, u_4))$, let $P$ be the subpath of $C_3(v_2, u_4)$ with ends $v, v_3$.
Then $C_2[u_2, u_3] \cup P_3 \cup P \cup uv \cup C_1[u, u_2] \cup P_2 \cup C_3[u_4, v_2]$ forms a kite-linkage of $u_2\com u_1\com u_3\com u_4$, a contradiction.
Otherwise, $v \in V(C_3[u_4, v_1))$, so $C_2[u_2, u_3] \cup P_3 \cup C_3[v_1, v_3] \cup P_1 \cup C_1[u_1, u] \cup uv \cup C_3[u_4, v]$ forms a kite-linkage of $u_2\com u_1\com u_3\com u_4$, a contradiction again.
\end{proof}

The proof of Theorem~\ref{kite} will now follow immediately from Theorem~\ref{flowerthm} and the next Theorem~\ref{thm:finalstep}.

\begin{theorem}\label{thm:finalstep}
If $G$ is $7$-connected and contains a $(u_1\com u_2\com u_3\com u_4)$-flower for some four vertices $u_2\com u_1\com u_3\com u_4 \in V(G)$, then $G$ contains a walk through $u_2, u_1, u_3, u_2, u_4$ in order.
\end{theorem}

\begin{proof}
Suppose otherwise.
Let $F = (C_1\com C_2\com C_3\com P_1\com P_2\com P_3)$ be a $(u_1\com u_2\com u_3\com u_4)$-flower as in Proposition~\ref{prop:Conn}.
Then $H = G - (V(C_1) \cup V(C_2))$ is $3$-connected.
We may assume that $H$ is planar, for otherwise $\mathcal{A} \ne \emptyset$ and by Lemma~\ref{prop}, $N(A)$ is a separating set in $G$ of order at most $3$, contradicting that $G$ is $7$-connected.
Since $G$ is $7$-connected, $\delta(G) \ge 7$.
Thus by Lemmas~\ref{discharge} and~\ref{prop}, there exists an edge $uv$ in $C_3[v_1, v_3] - v_2$ such that $d(u) + d(v) \le 7$ in $H$.
By symmetry, say $u, v \in V(C_3[v_1,v_2))$ with $u$ closer to $v_2$.
Let $\{R_1, R_2\}$ be a special $(u_1, u_2, C_1, \{u, v\})$-separating pair in $G$.
Let $X = \pint(R_1) \cup \pint(R_2) \cup \{u,v\}$, $T = \pend(R_1) \cup \pend(R_2)$, and $Y = V(G) \setminus (X \cup T)$.
We next show that $\pint(R_1) \cup \pint(R_2) \ne \emptyset$.
For otherwise, since $G$ is $7$-connected, $d(u) + d(v) \ge 7$ in $V(C_1)$, so at most one $r_i^j$ is not a neighbor of both $u, v$, where $r_i^j$ is the end of $R_i$ closer to $u_j$.
By the symmetry of $R_1$ and $R_2$, we may assume that $ur_1^1, ur_1^2, vr_1^1, vr_1^2 \in E(G)$.
Then $C_3[u_4, v] \cup vr_1^2 \cup C_1[r_1^2, r_1^1] \cup ur_1^1 \cup C_3[u, v_3] \cup P_3 \cup C_2[u_2, u_3]$ forms a kite-linkage of $u_2\com u_1\com u_3\com u_4$, a contradiction.
Now since $|T| \le 6$ and $G$ is $7$-connected, there exist $x \in X$ and $y \in Y$ such that $xy \in E(G)$.
By the definition of separating pair and symmetry of $R_1$ and $R_2$, we have $x \in \pint(R_1), y \in V(C_3[v_1,v_2)) \setminus \{v_2\}$.
So either $y \in V(C_3[v_1, v))$ or $y \in V(C_3(u, v_2))$.
Since $\{R_1, R_2\}$ is also a special $(u_2, u_1, C_1, \{u, v\})$-separating pair if we reverse the orientation of $C_1$, by Lemma~\ref{lem:ConnSep}(i), there is an edge $a_i a_i' \in E(G)$ so that $a_i \in \{u, v\}, a_i' \in V(C_1)$, and $G[V(C_1)]$ contains a path $Q_i$ with ends $x$ and $a_i'$ going through $x, u_{3 - i}, u_i, a_i'$ in order, for $i \in [2]$.
If $y \in V(C_3[v_1, v))$, then $C_3[u_4, y] \cup yx \cup Q_1 \cup a_1' a_1 \cup C_3[a_1, v_3] \cup P_3 \cup C_2[u_3, u_2]$ forms a kite-linkage of $u_2\com u_1\com u_3\com u_4$.
If $y \in V(C_3(u, v_2))$, then $C_3[u_4, a_2] \cup a_2 a_2' \cup Q_2 \cup xy \cup C_3[y, v_3] \cup P_3 \cup C_2[u_3, u_2]$ forms a kite-linkage of $u_2\com u_1\com u_3\com u_4$, in both cases a contradiction.
\end{proof}

\section{Final remarks}

It is well-known that a $3$-connected graph $G$ has a cycle containing any three given vertices in $G$.
The $2$-connected graphs which do not have this property are characterized in~\cite{WM67}.
We mentioned that fat-triangle-linkage is a generalization of this property.
It is interesting to characterize the $(k - 1)$-connected graphs that are not $F_{k_1, k_2, k_3}$-linked, where $k_1 + k_2 + k_3 = k$.

We showed that every $7$-connected graph with a specified flower is kite-linked.
The requirement of $8$-connectivity is only used to find the desired flower.
More specifically, we need $8$-connectivity in the proof of Theorem~\ref{flowerthm} in order to find a fat-triangle disjoint from a certain vertex.
If one can show the existence of flowers in $7$-connected graphs, then our results in this paper would show $f(P_4^+) = 7$.
A characterization as mentioned above may be helpful to show the existence of such flowers in $7$-connected graphs.

The proof for $f(P_4) = 7$ uses the characterization of non-$P_4$-linked graphs, which is quite complicated.
By using similar ideas as in this paper, we can find a self-contained and much simpler proof.
Here we sketch the steps of the proof.
Let $G$ be a $7$-connected graph, and let $u_1, u_2, u_3, u_4 \in V(G)$.
Suppose that we cannot find a path through $u_1, u_2, u_3, u_4$ in order.
Firstly, we can find the structure depicted in Figure~\ref{fig:mushroom} in $G$.
The proof of this is similar to the proof of Lemma~\ref{lem:flower} and Theorem~\ref{flowerthm} except that we start with a $(3, 3, 0)$-fat-triangle here.
We can then prove properties of this structure in a fashion similar to Proposition~\ref{prop:Conn}.
The remainder of the proof is similar to what we have done in Section~\ref{sec:final}.

\begin{figure}[H]
\includegraphics[scale=0.65]{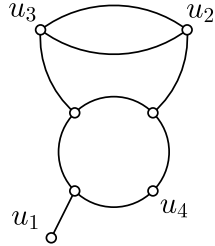}
\caption{Intermediate structure used to find a $P_4$-linkage}
\label{fig:mushroom}
\end{figure}

\bibliographystyle{plain}
\bibliography{GYUa}

\end{document}